\documentclass[11pt]{article}
\usepackage{amssymb,amsmath}
\usepackage[mathscr]{eucal}
\usepackage[cm]{fullpage}
\usepackage[english]{babel}
\usepackage[latin1]{inputenc}
\usepackage{color}

\def\dom{\mathop{\mathrm{Dom}}\nolimits}
\def\im{\mathop{\mathrm{Im}}\nolimits}
\def\N{\mathbb N}
\def\POI{\mathcal{POI}} 
\def\PO{\mathcal{PO}}
\def\I{\mathcal{I}} 
\def\G{\mathcal{G}} 
\def\R{\mathscr{R}}
\def\L{\mathscr{L}}
\def\J{\mathscr{J}}
\def\H{\mathscr{H}}
\def\Sym{\mathcal{S}} 
\def\PT{\mathcal{PT}}
\def\T{\mathcal{T}}
\def\Ord{\mathcal{O}}
\def\pv#1{\ensuremath{{\sf#1}}}
\def\no{\sqsubseteq} 
\newtheorem{theorem}{Theorem}[section]
\newtheorem{proposition}[theorem]{Proposition}
\newtheorem{corollary}[theorem]{Corollary}
\newtheorem{lemma}[theorem]{Lemma}

\newcommand{\qed}{{\unskip\nobreak
\hfil\penalty50\hskip .001pt \hbox{}
          \nobreak\hfil
          \vrule height 1.2ex width 1.1ex depth -.1ex
           \parfillskip=0pt\finalhyphendemerits=0\medbreak}}
\newenvironment{proof}{
            \begin{trivlist}\item[\hskip%
            \labelsep{\bf \noindent Proof.}]
              }
              {
            \hfill\qed\rm\end{trivlist}
              }

\newcommand{\lastpage}{\addresss}

\newcommand{\addresss}{\small \sf  
\noindent{\sc Rita Caneco}, 
Departamento de Matem\'atica, 
Faculdade de Ci\^encias e Tecnologia, 
Universidade Nova de Lisboa, 
Monte da Caparica, 
2829-516 Caparica, 
Portugal; 
e-mail: r.caneco@campus.fct.unl.pt

\medskip 

\noindent{\sc V\'\i tor H. Fernandes}, 
Departamento de Matem\'atica, 
Faculdade de Ci\^encias e Tecnologia, 
Universidade Nova de Lisboa, 
Monte da Caparica, 
2829-516 Caparica, 
Portugal; 
also: 
Centro de Matem\'atica e Aplica\c{c}\~oes, 
Faculdade de Ci\^encias e Tecnologia, 
Universidade Nova de Lisboa, 
Monte da Caparica, 
2829-516 Caparica, 
Portugal;
e-mail: vhf@fct.unl.pt

\medskip

\noindent{\sc Teresa M. Quinteiro}, 
Instituto Superior de Engenharia de Lisboa, 
Rua Conselheiro Em\'\i dio Navarro 1, 
1950-062 Lisboa, 
Portugal; 
also: 
Centro de Matem\'atica e Aplica\c{c}\~oes, 
Faculdade de Ci\^encias e Tecnologia, 
Universidade Nova de Lisboa, 
Monte da Caparica, 
2829-516 Caparica, 
Portugal;
e-mail: tmelo@adm.isel.pt 
}

\title{Ranks and presentations of some normally ordered inverse semigroups}

\author{Rita Caneco, V\'\i tor H. Fernandes\footnote{This work was partially supported by the  Funda\c{c}\~ao para a Ci\^encia e a Tecnologia (Portuguese Foundation for Science and Technology) through the project UID/MAT/00297/2019 (Centro de Matem\'atica e Aplica\c{c}\~oes), and of Departamento de Matem\'atica da Faculdade de Ci\^encias e Tecnologia da Universidade Nova de Lisboa.}~ and 
Teresa M. Quinteiro\footnote{This work was partially supported by the  Funda\c{c}\~ao para a Ci\^encia e a Tecnologia (Portuguese Foundation for Science and Technology) through the project UID/MAT/00297/2019 (Centro de Matem\'atica e Aplica\c{c}\~oes), and of \'Area Departamental Matem\'atica do  Instituto Superior de Engenharia de Lisboa.}
}

\begin{document}

\maketitle

\begin{abstract}
In this paper we compute the rank and exhibit a presentation for the monoids of all $P$-stable and $P$-order preserving partial permutations on a finite set $\Omega$, with $P$ an ordered uniform partition of $\Omega$. These (inverse) semigroups constitute a natural class of generators of the pseudovariety of inverse semigroups  $\pv{NO}$ of all normally ordered (finite) inverse semigroups. 
\end{abstract}

\medskip



\section{Introduction and preliminaries} 

Let $\Omega$ be a set. We denote by $\PT(\Omega)$ the monoid (under composition) of all 
partial transformations on $\Omega$, by $\T(\Omega)$ the submonoid of $\PT(\Omega)$ of all 
full transformations on $\Omega$, by $\I(\Omega)$ 
the \textit{symmetric inverse semigroup} on $\Omega$, i.e. 
the inverse submonoid of $\PT(\Omega)$ of all 
partial permutations on $\Omega$, 
and by $\Sym(\Omega)$ the \textit{symmetric group} on $\Omega$, 
i.e. the subgroup of $\PT(\Omega)$ of all 
permutations on $\Omega$. 
If $\Omega$ is a finite set with $n$ elements ($n\in\N$), 
say $\Omega=\Omega_n=\{1,\ldots,n\}$, we denote 
$\PT(\Omega)$, $\T(\Omega)$, $\I(\Omega)$ and $\Sym(\Omega)$ simply by $\PT_n$, $\T_n$, $\I_n$ and $\Sym_n$, respectively. 

Now, consider a linear order $\le$ on $\Omega_n$, e.g the usual order.  We say that a transformation $\alpha\in\PT_n$ is 
\textit{order preserving} if $x\leq y$ implies $x\alpha\leq y\alpha$, for all $x,y \in \dom(\alpha)$. 
Denote by $\PO_n$ the submonoid of $\PT_n$ of all order preserving partial transformations, 
by $\Ord_n$ the submonoid of $\T_n$ of all
order preserving full transformations of $\Omega_n$ and by $\POI_n$ the inverse submonoid of $\I_n$ of all order preserving partial permutations of $\Omega_n$. 

\smallskip

A pseudovariety of [inverse] semigroups is a class of finite [inverse] semigroups 
closed under homomorphic images of [inverse] subsemigroups and finitary direct products. 

\smallskip

In the ``Szeged International Semigroup Colloquium" (1987) 
J.-E. Pin asked for an {\it effective} description of the pseudovariety 
(i.e. an algorithm to decide whether or not a finite 
semigroup belongs to the pseudovariety) of semigroups $\pv{O}$ generated 
by the semigroups $\Ord_n$, with $n\in\N$. 
Despite, as far as we know, this question is still open, some progresses have been made. 
First, Higgins \cite{Higgins:1995} proved 
that $\pv{O}$ is self-dual and does not contain all 
$\mathscr{R}$-trivial semigroups (and so $\pv{O}$ is properly 
contained in $\pv{A}$, the pseudovariety 
of all finite aperiodic semigroups), 
although every finite band belongs to $\pv{O}$. 
Next, Vernitskii and Volkov \cite{Vernitskii&Volkov:1995} generalized Higgins's result by showing that every 
finite semigroup whose idempotents form an ideal is in $\pv{O}$ and, 
in \cite{Fernandes:1997}, Fernandes proved that the pseudovariety 
of semigroups $\pv{POI}$ generated by the semigroups 
$\POI_n$, with $n\in\N$, is a (proper) subpseudovariety of $\pv{O}$. 
On the other hand, Almeida and Volkov \cite{Almeida&Volkov:1998} showed that 
the interval $[\pv{O}, \pv{A}]$ of the lattice of all pseudovarieties 
of semigroups has the cardinality of the continuum and 
Repnitski\u{\i} and Volkov \cite{Repnitskii&Volkov:1998} proved 
that $\pv{O}$ is not finitely based. 
Another contribution to the resolution of 
Pin's problem was given by Fernandes \cite{Fernandes:2002} who  
showed that $\pv{O}$ contains all semidirect products 
of a chain (considered as a semilattice) by  
a semigroup of injective order preserving 
partial transformations on a finite chain. 
This result was later generalized by Fernandes and Volkov \cite{Fernandes&Volkov:2010} 
for semidirect products of a chain by any semigroup from $\pv{O}$. 

\smallskip 

The inverse counterpart of Pin's problem can be formulated by asking for 
an effective description of the pseudovariety of inverse semigroups 
\pv{PCS} generated by $\{\POI_n\mid n\in\N\}$.  
In \cite{Cowan&Reilly:1995} Cowan and Reilly 
proved that $\pv{PCS}$ is properly contained in $\pv{A}\cap\pv{Inv}$ (being $\pv{Inv}$ the class of all inverse semigroups) 
and also that the interval $[\pv{PCS}, \pv{A}\cap\pv{Inv}]$
of the lattice of all pseudovarieties of inverse semigroups has the
cardinality of the continuum. From Cowan and Reilly's results
it can be deduced that a finite inverse semigroup with $n$ elements
belongs to $\pv{PCS}$ if and only if it can be embedded into the
semigroup $\POI_n$. This is in fact an effective description of
$\pv{PCS}$. On the other hand, in \cite{Fernandes:1998} Fernandes 
introduced the class $\pv{NO}$ of all normally ordered inverse semigroups. 
This notion is deeply related with the Munn representation of an inverse semigroup $M$, an 
idempotent-separating homomorphism that may be defined by   
$$
\begin{array}{rccccccc}
\phi: & M & \rightarrow & \I(E)&&&\\
        & s & \mapsto & \phi_s:&Ess^{-1}&\rightarrow&Es^{-1}s\\
         &&&                    & e & \mapsto &s^{-1}es~,
\end{array}
$$
with $E$ the semilattice of all idempotents of $M$. 
A finite inverse semigroup $M$ is said to be \textit{normally ordered} if 
there exists a linear order $\no$ in the semilattice $E$ of the idempotents of $M$ 
preserved by all partial permutations $\phi_s$
(i.e. for $e, f \in Ess^{-1}$, $e\no f$ implies $e\phi_s\no f\phi_s$), with $s\in M$. 
It was proved in \cite{Fernandes:1998} that $\pv{NO}$ is a
pseudovariety of inverse semigroups and also that 
the class of all fundamental normally ordered inverse semigroups 
consists of all aperiodic normally ordered inverse semigroups. 
Moreover, Fernandes showed that $\pv{PCS}=\pv{NO}\cap\pv{A}$, 
giving in this way a Cowan and Reilly 
alternative (effective) description of $\pv{PCS}$.  
In fact, this also led Fernandes \cite{Fernandes:1998} to 
the following refinement of Cowan and Reilly's description of $\pv{PCS}$: 
a finite inverse semigroup with $n$ {\it idempotents} belongs
to $\pv{PCS}$ if and only if it can be embedded into $\POI_n$. 
Another refinement (in fact, the best possible) was also given by Fernandes \cite{Fernandes:2008}, 
by considering only join irreducible idempotents. 
Notice that, in \cite{Fernandes:1998} it was also proved that $\pv{NO}=\pv{PCS}\vee\pv{G}$ 
(the join of $\pv{PCS}$ and $\pv{G}$, the pseudovariety of all groups). 

\medskip 

Now, let $\Omega$ be a finite set. 

An \textit{ordered partition} of $\Omega$ is a partition of $\Omega$ endowed with a linear order. 
By convention, whenever we take an ordered partition $P=\{X_i\}_{i=1,\dots,k}$ of $\Omega$, we will assume that $P$ is the chain 
$\{X_1<X_2<\cdots <X_k\}$. 
We say that $P=\{X_i\}_{i=1,\dots,k}$ is an \textit{uniform partition} of $\Omega$ if $|X_i|=|X_j|$ for all $i,j\in\{1,\ldots,k\}$.

Let $P=\{X_i\}_{i=1,\dots,k}$ be an ordered partition of $\Omega$ and, for each $x\in \Omega$, denote by $i_x$ the integer $i\in \{1,\ldots,k\}$ such that $x\in X_i$. Let $\alpha$ be a partial transformation on $\Omega$. We say that $\alpha$ is:
\begin{description}
\item $P$-\textit{stable} if $X_{i_x}\subseteq \dom(\alpha)$ and $X_{i_x}\alpha=X_{i_{x\alpha}}$, for all $x\in \dom(\alpha)$; and 
\item $P$-\textit{order preserving} if $i_x\le i_y$ implies $i_{x\alpha}\le i_{y\alpha}$, for all $x,y\in \dom(\alpha)$
\end{description}
(where $\le$ denotes the usual order on $\{1,\ldots,k\}$). 

Denote by $\POI_{\Omega,P}$ the set of all $P$-stable and  $P$-order preserving partial permutations on $\Omega$. 

Notice that, the identity mapping belongs to $\POI_{\Omega,P}$ and it is easy to check that $\POI_{\Omega,P}$ is an inverse submonoid of $\I(\Omega)$. 
Observe also that if $P$ is the trivial partition of $\Omega_n$, i.e. $P=\{\{i\}\}_{i=1,\ldots,n}$, then $\POI_{\Omega_n,P}$ coincides with $\POI_n$ and, 
on the other hand, if $P=\{\Omega_n\}$ (the universal partition of $\Omega_n$) then $\POI_{\Omega_n,P}$ is exactly the symmetric group $\Sym_n$, for $n\in\N$. 

These monoids, considered for the first time by Fernandes in \cite{Fernandes:1998}, were inspired by the work of Almeida and Higgins 
\cite{Almeida&Higgins:1997}, despite they are quite different from the ones considered by these two last authors. 
The main relevance of the monoids $\POI_{\Omega,P}$ lies 
in the fact that they constitute a family of generators of the pseudovariety $\pv{NO}$ 
of inverse semigroups. More precisely, Fernandes proved in \cite[Theorem 4.4]{Fernandes:1998} that $\pv{NO}$ is the class of all inverse subsemigroups (up to an isomorphism) of semigroups of the form $\POI_{\Omega,P}$.  
In fact, by the proof of \cite[Theorem 4.4]{Fernandes:1998}, it is clear that it suffices to consider semigroups of the form 
$\POI_{\Omega,P}$, with $P$ a uniform partition of $\Omega$, i.e. we also have the following result: 

\begin{theorem} \label{old} 
The class $\pv{NO}$ is the pseudovariety of inverse semigroups generated by all semigroups of the form $\POI_{\Omega,P}$, 
where $\Omega$ is a finite set and $P$ is an ordered uniform partition of $\Omega$. 
\end{theorem} 

\medskip 

Let $n\in\N$. An \textit{ordered partition} of $n$ is a non-empty sequence of positive integers whose elements sum to $n$. 

Let $\pi=(n_1,\ldots,n_k)$ be an ordered partition of $n$. Define the ordered partition $P_{\pi}$ of $\Omega_n$ as being the partition into intervals 
$P_{\pi}=\{I_i\}_{i=1,\dots,k}$ of $\Omega_n$ (endowed with the usual order), where
$$
I_1=\{1,\dots, n_1\} \quad\text{and}\quad I_i = \{ n_1+\cdots+n_{i-1}+1,\ldots,n_1+\cdots+n_i\}, ~\mbox{for $2\le i\le k$.}
$$
Notice that $\pi=(|I_1|,\ldots,|I_k|)$. 

Next, we show that $\Omega_n$ and its partitions into intervals allow us to construct, up to an isomorphism, 
all monoids of type $\POI_{\Omega,P}$, with $\Omega$ a set with $n$ elements and $P$ and ordered partition of $\Omega$. 

\begin{theorem}\label{iso} 
Let $\Omega$ be a set with $n$ elements and let $P=\{X_i\}_{i=1,\dots,k}$ be an ordered partition of $\Omega$. Then, 
being $\pi$ the ordered partition $(|X_1|,\ldots,|X_k|)$ of $n$, 
the monoids $\POI_{\Omega,P}$ and $\POI_{\Omega_n,P_{\pi}}$ are isomorphic. 
\end{theorem}

\begin{proof}
Let $\pi=(n_1,\ldots,n_k)$ and $P_{\pi}=\{I_i\}_{i=1,\dots,k}$. Then $|X_i|=|I_i|$, for all $i\in\{1,\ldots,k\}$, and so we may consider a bijection 
$\sigma:\Omega\longrightarrow\Omega_n$ such that $X_i\sigma=I_i$, for all $i\in\{1,\ldots,k\}$. Hence, it is clear that the mapping 
$$
\begin{array}{rccc}
\Psi:&\POI_{\Omega,P}&\longrightarrow&\I_n\\
   & \alpha &\longmapsto & \sigma^{-1}\alpha\sigma
\end{array} 
$$
is a homomorphism of monoids such that $\alpha\Psi$ is $P_{\pi}$-stable, for all $\alpha\in\POI_{\Omega,P}$. 

For each $x\in \Omega$, denote by $i_x$ the integer $i\in \{1,\ldots,k\}$ such that $x\in X_i$ and, for each $a\in \Omega_n$, denote by $i_a$ the integer $i\in \{1,\ldots,k\}$ such that $a\in I_i$. Clearly, $i_x=i_{x\sigma}$ and $i_a=i_{a\sigma^{-1}}$, for all $x\in \Omega$ and $a\in\Omega_n$. 

Let $\alpha\in\POI_{\Omega,P}$. Notice that $\dom(\alpha\Psi)=(\dom(\alpha))\sigma$ (and $\im(\alpha\Psi)=(\im(\alpha))\sigma$). Take $a,b\in\dom(\alpha\Psi)$ 
such that $i_a\le i_b$. Then $i_{a\sigma^{-1}}\le i_{b\sigma^{-1}}$ and, since $a\sigma^{-1}, b\sigma^{-1}\in\dom(\alpha)$, we have 
$i_{(a\sigma^{-1})\alpha}\le i_{(b\sigma^{-1})\alpha}$. Hence
$$
i_{a(\alpha\Psi)}=i_{a\sigma^{-1}\alpha\sigma}=i_{a\sigma^{-1}\alpha}\le i_{b\sigma^{-1}\alpha}=i_{b\sigma^{-1}\alpha\sigma}=i_{b(\alpha\Psi)}, 
$$
which proves that $\alpha\Psi\in\POI_{\Omega_n,P_{\pi}}$. 

Thus, we may consider $\Psi$ as a homomorphism of monoids from $\POI_{\Omega,P}$ into $\POI_{\Omega_n,P_{\pi}}$. 
Analogously, we build a homomorphism of monoids $\Phi: \POI_{\Omega_n,P_{\pi}}\longrightarrow\POI_{\Omega,P}$ by defining $\beta\Phi=\sigma\beta\sigma^{-1}$, 
for each $\beta\in\POI_{\Omega_n,P_{\pi}}$. 
Clearly, $\Psi$ and $\Phi$ are mutually inverse mappings and so they are isomorphisms of monoids, as required.  
\end{proof}

\smallskip

Now, let $k,m\in \N$ be such that $n=km$. Let $\pi=(m,\ldots,m)\in\Omega_n^k$. Denote the uniform partition into intervals $P_{\pi}$ of $\Omega_n$ by $P_{k\times m}$ 
(i.e. we have  $P_{k\times m}=\{I_i\}_{i=1,\dots,k}$, with  
$
I_i=\{(i-1)m+1,\ldots, im\},
$ 
for $i\in \{1,\ldots,k\}$) and denote the monoid $\POI_{\Omega_n,P_{k\times m}}$ by $\POI_{k\times m}$. Therefore, combining Theorems \ref{old} and \ref{iso}, we immediately obtain the following result: 

\begin{corollary} \label{motiv} 
The pseudovariety of inverse semigroups $\pv{NO}$ is generated by the class $\{\POI_{k\times m}\mid k,m\in \N\}$.
\end{corollary}

\smallskip 

This fact gave us the main motivation for the work presented in this paper, which is about the monoids $\POI_{k\times m}$, with $k,m\in\N$. 
The remaining of this paper is organized as follows. 
In Section \ref{rank} we calculate their sizes and ranks and
in Section \ref{pres} we construct presentations for them. 

\smallskip

For general background on Semigroup Theory and standard notation, we refer the reader to Howie's book \cite{Howie:1995}.
For general background on pseudovarieties and finite semigroups, 
we refer the reader to Almeida's book \cite{Almeida:1995}.  
All semigroups considered in this paper are finite. 

\section{Size and rank of $\POI_{k\times m}$}\label{rank}

Let $M$ be a monoid. Recall that the quasi-order $\le_\J$ is defined on $M$ as follows: for all $u,v\in M$, $u\le_\J v$ if and only if $MuM\subseteq MvM$. As usual, the $\J$-class of an element $u\in M$ is denoted by $J_u$ and a partial order relation $\le_\J$ is defined on the set $M/\J$ by $J_u\le_\J  J_v$ if and only if $u\le_\J v$. Given $u,s\in M$, 
we write $u<_\J v$ or $J_u<_\J  J_v$ if and only if $u\le_\J v$ and $(u,v)\notin \J$. 

Recall also that the \textit{rank} of a (finite) monoid $M$ is the minimum size of a generating set of $M$. 

\smallskip 

Let $P$ be an ordered partition of $\Omega$. Let $\alpha,\beta\in \POI_{\Omega,P}$. Since $ \POI_{\Omega,P}$ is an inverse submonoid of $\I(\Omega)$, we immediately have that $\alpha\R \beta$ if and only if $\dom(\alpha)=\dom(\beta)$ and that  
$\alpha\L \beta$ if and only if $\im(\alpha)=\im(\beta)$. 
If $P$ is uniform, it is easy to check also that $\alpha\J \beta$ if and only if $|\im(\alpha)|=|\im(\beta)|$ 
(see \cite[Proposition 5.2.2]{Fernandestese:1998}). 
In fact, more specifically, we have that $J_\alpha\le_\J J_\beta$ if and only if $|\im(\alpha)|\le |\im(\beta)|$. 

\smallskip 

Notice that $\POI_{n\times 1}$ is isomorphic to $\POI_n$, 
whose size is $\binom{2n}{n}$ and rank is $n$ (see \cite[Proposition 2.2]{Fernandes:1997} and \cite[Proposition 2.8]{Fernandes:2001}), 
and $\POI_{1\times n}$ is isomorphic to $\Sym_n$, whose size is well known to be $n!$ and rank is well known to be $2$, for $n\ge3$, and $1$, for $n\in\{1,2\}$.  

\smallskip 

From now on let $k,m\in \N$ be such that $k,m\ge 2$ and let $n=km$. 

\smallskip 

Now, we turn our attention to the $\J$-classes of $\POI_{k\times m}$. 
Let $\alpha\in \POI_{k\times m}$. Then $|\im(\alpha)|=im$, for some $0\le i\le k$. Hence 
$$
\POI_{k\times m}/\J=\{J_0<_\J J_1<_\J\cdots <_\J J_k \},
$$ 
where $J_i=\{\alpha\in \POI_{k\times m}\mid |\im(\alpha)|=im\}$, for $0\le i\le k$. 

\smallskip 

Let $t\in \{1,\ldots,k\}$. We write 
$$
\alpha=\left(\begin{array}{c|c|c|c}
I_{i_1}&I_{i_2}&\cdots&I_{i_t}\\
I_{j_1}&I_{j_2}&\cdots&I_{j_t}
\end{array}\right), 
$$
for all transformations $\alpha\in \POI_{k\times m}$ such that 
$\dom(\alpha)=I_{i_1}\cup I_{i_2}\cup\cdots\cup I_{i_t}$, 
$\im(\alpha)=I_{j_1}\cup I_{j_2}\cup\cdots\cup I_{j_t}$ and 
$I_{i_r}\alpha=I_{j_r}$, for $1\le r\le t$.  
In this case, we assume always that $1\le i_1<i_2<\cdots<i_t\le k$ and $1\le j_1<j_2<\cdots<j_t\le k$. 
Clearly, the set of all such transformations forms an $\H$-class of $\POI_{k\times m}$ contained in $J_t$. 

In particular, it is easy to check that the $\H$-class of the transformations of the form 
$$
\left(\begin{array}{c|c|c|c}
I_{i_1}&I_{i_2}&\cdots&I_{i_t}\\
I_{i_1}&I_{i_2}&\cdots&I_{i_t}
\end{array}\right)
$$
constitutes a group isomorphic to $\Sym_m^t$ and so it has ${(m!)}^t$ elements. 

On the other hand, since there are $\binom{k}{t}$ distinct possibilities for domains (and images) of the transformations of $J_t$, 
we deduce that $|J_t|=\binom{k}{t}^2{(m!)}^t$. 

Thus, we have:

\begin{proposition}
For $k,m\ge2$, the monoid $\POI_{k\times m}$ has $\sum_{t=0}^k \binom{k}{t}^2{(m!)}^t$ elements.
\end{proposition}

\medskip

Next, let 
$
\psi:\POI_{k}\longrightarrow\POI_n  
$
be the mapping defined by 
$$
\dom(\theta\psi)=\cup\{I_i\mid i\in\dom(\theta)\}\quad\text{and}\quad\im(\theta\psi)=\cup\{I_i\mid i\in\im(\theta)\},
$$
for all $\theta\in\POI_k$. 
Notice that, if 
$$
\theta=\left(\begin{array}{cccc}
i_1&i_2&\cdots&i_t\\
j_1&j_2&\cdots&j_t
\end{array}\right)\in \POI_k, 
$$ 
with $1\le t\le k$, $1\le i_1<i_2<\cdots<i_t\le k$ and $1\le j_1<j_2<\cdots<j_t\le k$, then 
$$
\theta\psi=\left(\begin{array}{c|c|c|c}
I_{i_1}&I_{i_2}&\cdots&I_{i_t}\\
I_{j_1}&I_{j_2}&\cdots&I_{j_t}
\end{array}\right)\in \POI_n\cap\POI_{k\times m}. 
$$
Moreover, it is a routine matter to show that $\im(\psi)=\POI_n\cap\POI_{k\times m}$ and $\psi$ is an injective homomorphism of monoids.
 
\smallskip

Let 
$$
{x}_0=\left(\begin{array}{cccc}
2&\cdots&k-1&k\\
1&\cdots&k-2&k-1
\end{array}\right), ~ 
{x}_i=\left(\begin{array}{ccccccc}
1&\cdots&k-i-1&k-i&k-i+2&\cdots&k\\
1&\cdots&k-i-1&k-i+1&k-i+2&\cdots&k
\end{array}\right),~ 1\le i\le k-1, 
$$
and take $\bar{x}_i={x}_i\psi$, for $0\le i\le k-1$. Observe that $\bar{x}_0,\bar{x}_1,\ldots, \bar{x}_{k-1}$ are order preserving and $P$-order preserving transformations such that 
$$
\bar{x}_0=\left(\begin{array}{c|c|c|c}
I_2&\cdots&I_{k-1}&I_k\\
I_1&\cdots&I_{k-2}&I_{k-1}
\end{array}\right), ~ 
\bar{x}_i=\left(\begin{array}{c|c|c|c|c|c|c}
I_1&\cdots&I_{k-i-1}&I_{k-i}&I_{k-i+2}&\cdots&I_k\\
I_1&\cdots&I_{k-i-1}&I_{k-i+1}&I_{k-i+2}&\cdots&I_k
\end{array}\right),~ 1\le i\le k-1.
$$
Since $\POI_{k}$ is generated by
$\{x_0,x_1,\ldots,x_{k-1}\}$ (see \cite{Fernandes:2001}) and $\psi$ is a homomorphism, 
then $\bar{X}=\{\bar x_0,\bar x_1,\ldots,\bar x_{k-1}\}$ is a generating set for $\im(\psi)=\POI_n\cap\POI_{k\times m}$. 

\medskip

Next, recall that is well known that $\Sym_m$ is generated by the permutations $a=(1\ 2)$ and $b=(1\ 2 \cdots m)$. 
Take $c=ab=(1\ 3\ 4 \cdots m)$. 
Thus, since $a=cb^{m-1}$, it is clear that $\Sym_{m}$ is also generated by the permutations $b$ and $c$. 
Let 
$$
a_{i}=(1,\ldots,1,a,1,\ldots,1),~ b_{i}=(1,\ldots,1,b,1,\ldots,1)~ \text{and} ~ c_{i}=(1,\ldots,1,c,1,\ldots,1),
$$
where $a$, $b$ and $c$ are in the position $i$, for $1\le i\le k$ (and $1$ denoting the identity of $\Sym_{m}$). 
Clearly, 
$$
\{ a_{1},a_{2},\ldots,a_{k}, b_{1},b_{2},\ldots,b_{k}\}~ \text{and} ~\{ b_{1},b_{2},\ldots,b_{k}, c_{1},c_{2},\ldots,c_{k}\}
$$
are generating sets of the direct product $\Sym_{m}^{k}$. 

Let $d_{i}=b_{i}c_{i+1}$, for $1\le i\le k-1$, and $d_{k}=b_{k}c_{1}$. 
For $1\le i\le k$, we have $b_{i}^{m}=c_{i}^{m-1}=1$, whence $b_{i}^{{(m-1)}^{2}}=b_{i}$ and $c_{i}^{m}=c_{i}$. 
Moreover, since $b_{i}c_{j}=c_{j}b_{i}$, for $1\le i,j\le k$ and $i\ne j$, it is easy to check that $c_{1}=d_{k}^{m}$, 
$c_{i+1}=d_{i}^{m}$, for $1\le i\le k-1$, and $b_{i}=d_{i}^{{(m-1)}^2}$, for $1\le i\le k$. 
Therefore
$$
\{ d_{1},d_{2},\ldots,d_{k}\}
$$
is also a generating set of $\Sym_{m}^{k}$. 
Observe that, as $m,k\ge2$, the rank of $\Sym_{m}^{k}$ is $k$ (for instance, see \cite{Wiegold:1972}).

\medskip  

Let $\G_{k\times m}$ be the group of units of $\POI_{k\times m}$, i.e. 
$\G_{k\times m}=\{\alpha\in\POI_{k\times m}\mid |\im(\alpha)|=n\}=\Sym_n\cap\POI_{k\times m}$.   
We have a natural isomorphism
$$
\begin{array}{ccc}
\Sym_{m}^{k}&\longrightarrow&\G_{k\times m}\\
   z &\longmapsto & \bar z, 
\end{array} 
$$
defined by $x\bar z=(x-(i-1)m)z_i+(i-1)m$, for $x\in I_i$, $1\le i\le k$, and $z=(z_1,z_2,\ldots,z_k)\in\Sym_{m}^{k}$. 

\medskip

Let $\bar A=\{\bar{a}_{1},\bar{a}_{2},\ldots,\bar{a}_{k}\}$, $\bar B=\{\bar{b}_{1},\bar{b}_{2},\ldots,\bar{b}_{k}\}$, 
$\bar C=\{\bar{c}_{1},\bar{c}_{2},\ldots,\bar{c}_{k}\}$ and $\bar D=\{\bar{d}_{1},\bar{d}_{2},\ldots,\bar{d}_{k}\}$.
By the above observations, $\bar A\cup \bar B$, $\bar B\cup \bar C$ and $\bar D$ are three generating sets of $\G_{k\times m}$. 
Moreover, the number of elements of $\bar D$ is $k$, which is precisely the rank of $\G_{k\times m}$.  

\begin{proposition}\label{gerprop}
For $k,m\ge 2$, the monoid $\POI_{k\times m}$ is generated by $\bar{X}\cup\G_{k\times m}$. 
\end{proposition}

\begin{proof}
Let $\alpha\in\POI_{k\times m}$ be a nonempty transformation (notice that, clearly, $\langle{\bar X}\rangle$ contains the empty transformation) and suppose that 
$\alpha=\left(\begin{array}{c|c|c|c}
I_{i_1}&I_{i_2}&\cdots&I_{i_t}\\
I_{j_1}&I_{j_2}&\cdots&I_{j_t}
\end{array}\right)$, 
with $1\le t\le k$. 

Take $\theta=\left(\begin{array}{cccc}
i_1&i_2&\cdots&i_t\\
j_1&j_2&\cdots&j_t
\end{array}\right) \in \POI_k$ and  let $\bar\alpha=\theta\psi$. Then $\bar\alpha\in \langle{\bar X}\rangle$. 

On the other hand, define $\gamma\in\T_n$ by $x\gamma=(x+(i_r-j_r)m)\alpha$, for $x\in I_{j_r}$, $1\le r\le t$,  
and $x\gamma=x$, for $x\in\Omega_n\setminus\im(\alpha)$. 
Then $\gamma\in\G_{k\times m}$ and it is a routine matter to check that $\alpha=\bar\alpha\gamma$, which proves the result. 
\end{proof}

It follows immediately: 

\begin{corollary}\label{ger} 
For $k,m\ge 2$, $\bar A\cup \bar B \cup\bar X$, $\bar B\cup \bar C\cup\bar X$ and $\bar D\cup\bar X$
are generating sets of $\POI_{k\times m}$. 
\end{corollary}

Notice that, in particular,  $\bar{D}\cup\bar X$ is a generating set of $\POI_{k\times m}$ with $2k$ elements. 

\begin{proposition} 
For $k,m\ge 2$, the rank of $\POI_{k\times m}$ is $2k$.
\end{proposition}

\begin{proof}
Let $L$ be a generating set of $\POI_{k\times m}$. 

Take a transformation $\alpha\in\POI_{k\times m}$ of rank $m(k-1)$. Then, $\im(\alpha)=\Omega_n\setminus I_j$, for some $1\le j\le k$. 
Let $\alpha_1,\alpha_2,\ldots,\alpha_t\in L$ be such that $\alpha=\alpha_1\alpha_2\cdots\alpha_t$. 
Hence, for $1\le i\le t$, the rank of $\alpha_i$ is either $mk$ or $m(k-1)$. 
Moreover, at least one of the transformations $\alpha_1,\alpha_2,\ldots,\alpha_t$ must have rank equal to $m(k-1)$. 
Let 
$$
p=\max\{i\in\{1,\ldots, t\}\mid \alpha_i \textrm{ has rank equal to } m(k-1)\}.
$$  
Then $\xi=\alpha_{p+1}\cdots\alpha_t$ has rank $km$ (here $\xi $ denotes the identity, if $p=t$) and so 
$\xi= \left(\begin{array}{c|c|c|c}
I_1&I_2&\cdots&I_k\\
I_1&I_2&\cdots&I_k
\end{array}\right)\in\G_{k\times m}$. 
Thus $\im(\alpha)=(\im(\alpha_1\alpha_2\cdots\alpha_p))\xi=\im(\alpha_1\alpha_2\cdots\alpha_p)\subseteq \im(\alpha_p)$. 
As $|\im(\alpha)|=|\im(\alpha_p)|$, we deduce that $\im(\alpha_p)=\im(\alpha)$. 
Therefore  $L$ contains at least one element whose image is $\Omega_n\setminus I_j$, for each $1\le j\le k$, 
and so $L$ must contain at least $k$ elements of rank $m(k-1)$. 

Next, we consider the elements of $\POI_{k\times m}$ of rank $km$. 
They constitute the group of units $\G_{k\times m}$ of $\POI_{k\times m}$, which has rank $k$ (as observed above). 
Therefore, we must also have at least $k$ elements of rank $km$ in the generating set $L$.

Thus $|L|\ge 2k$. Since $\bar{D}\cup\bar X$ is a generating set of $\POI_{k\times m}$ with $2k$ elements, the result follows. 
\end{proof}

\section{Presentations for $\POI_{k\times m}$} \label{pres}

We begin this section by recalling some notions and facts on presentations. 
Let $A$ be a set and denote by $A^*$ the free monoid generated by
$A$. Usually, the set $A$ is called \textit{alphabet} and the elements of $A$ and $A^*$ are called \textit{letters} and 
\textit{words}, respectively.  
A \textit{monoid presentation} is an ordered pair $\langle A\mid
R\rangle$, where $A$ is an alphabet and $R$ is a subset of
$A^*\times A^*$. An element $(u,v)$ of $A^*\times A^*$ is called a
{\it relation} and it is usually represented by $u=v$. To avoid
confusion, given $u, v\in A^*$, we will write $u\equiv v$, instead
of  $u=v$, whenever we want to state precisely that $u$ and $v$
are identical words of $A^*$. A monoid $M$ is said to be {\it
defined by a presentation} $\langle A\mid R\rangle$ if $M$ is
isomorphic to $A^*/\rho_R$, where $\rho_R$ denotes the smallest
congruence on $A^*$ containing $R$. For more details see
\cite{Lallement:1979} or \cite{Ruskuc:1995}.

A direct method to find a presentation for a monoid
is described by the following well-known result (e.g.  see \cite[Proposition 1.2.3]{Ruskuc:1995}).  

\begin{proposition}\label{met} 
Let $M$ be a monoid generated by a set $A$ (also considered as an alphabet) and let $R\subseteq A^*\times A^*$.
Then $\langle A\mid R\rangle$ is a presentation for $M$ if and only
if the following two conditions are satisfied:
\begin{enumerate}
\item
The generating set $A$ of $M$ satisfies all the relations from $R$;  
\item 
If $u,v\in A^*$ are any two words such that 
the generating set $A$ of $M$ satisfies the relation $u=v$, then $u=v$ is a consequence of $R$.
\end{enumerate} 
\end{proposition}

Given a presentation for a monoid, a method to find a new
presentation consists in applying Tietze transformations. For a
monoid presentation $\langle A\mid R\rangle$, the 
(four) \emph{elementary Tietze transformations} are:

\begin{description}
\item-
Adding a new relation $u=v$ to $\langle A\mid R\rangle$,
provided that $u=v$ is a consequence of $R$;
\item-
Deleting a relation $u=v$ from $\langle A\mid R\rangle$,
provided that $u=v$ is a consequence of $R\backslash\{u=v\}$;
\item-
Adding a new generating symbol $b$ and a new relation $b=w$, where
$w\in A^*$;
\item- 
If $\langle A\mid R\rangle$ possesses a relation of the form
$b=w$, where $b\in A$, and $w\in(A\backslash\{b\})^*$, then
deleting $b$ from the list of generating symbols, deleting the
relation $b=w$, and replacing all remaining appearances of $b$ by
$w$.
\end{description}

The next result is also well-known (e.g. see \cite[Proposition 3.2.5]{Ruskuc:1995}): 

\begin{proposition}\label{tietze}
Two finite presentations define the same monoid if and only if one
can be obtained from the other by applying a finite number of elementary
Tietze transformations.  
\end{proposition}

Another tool that we will use is given by the following proposition (e.g. see \cite{Howie&Ruskuc:1995}): 

\begin{proposition}\label{dir}
Let $M$ and $N$ be two monoids defined by the monoid presentations $\langle A\mid R\rangle$ and $\langle B\mid S\rangle$, respectively.  Then the monoid presentation $\langle A, B\mid R,S, ab=ba, \, a\in A, b\in B\rangle$ defines the direct product $M\times T$. 
\end{proposition} 

\smallskip 

Our strategy for obtainning a presentation for $\POI_{k\times m}$ will use well-known presentations of $\Sym_{m}$ and $\POI_k$. 

\smallskip 

First, we consider the following (monoid) presentation of $\Sym_{m}$, with $m+1$ relations in terms of the generators $a$ and $b$ defined in the previous section:
$$
\langle a,b \mid a^{2}=b^{m}=(ba)^{m-1}=(ab^{m-1}ab)^{3}=(ab^{m-j}ab^{j})^{2}=1,\ 2\le j\le m-2 \rangle 
$$
(for instance, see \cite{Fernandes:2002b}). 
From this presentation, applying Tietze transformations, we can easily deduce the following presentation for $\Sym_{m}$, also with $m+1$ relations, in terms of the generators $b$ and $c$ (also defined in the previous section):
$$
\langle b,c \mid (cb^{m-1})^{2}=b^{m}=(bcb^{m-1})^{m-1}=(cb^{m-2}c)^{3}=(cb^{m-j-1}cb^{j-1})^{2}=1,\ 2\le j\le m-2 \rangle 
$$
(recall that $c=ab$ and $a=cb^{m-1}$). Notice that $c^{m-1}=1$. 

Next, we use these presentations of $\Sym_{m}$ for getting two presentations of $\Sym_{m}^{k}$.

\smallskip 

Consider the alphabets $A=\{a_i \mid 1\leq i\leq k\}$ and $B=\{b_{i}\mid 1\leq i\leq k\}$ (with $k$ letters each) and the set $R$ formed by the following $2k^{2}+(m-1)k$ monoid relations:
\begin{description}
\item[($R_1$)] $a_{i}^2=1$, $1\le i\le k$;
\item[($R_2$)] $b_{i}^m=1$, $1\le i\le k$ ;
\item[($R_3$)] $(b_{i}a_{i})^{m-1}=1$, $1\le i\le k$;
\item[($R_4$)] $(a_{i}b_{i}^{m-1}a_{i}b_{i})^{3}=1$, $1\le i\le k$; 
\item[($R_5$)] $(a_{i}b_{i}^{m-j}a_{i}b_{i}^{j})^{2}=1$, $2\leq j\leq m-2$, $1\le i\le k$;
\item[($R_6$)] $a_ia_j=a_ja_i$, $b_ib_j=b_jb_i$, $1\le i < j\le k$; \quad 
$a_{i}b_{j}=b_{j}a_{i}$, $1\le i,j\le k$, $i\ne j$.
\end{description}
Then, by Proposition \ref{dir}, the monoid $\Sym_{m}^{k}$ is defined by the presentation $\langle A,B\mid R\rangle$.

\smallskip

Now, consider the alphabet $C=\{c_i \mid 1\leq i\leq k\}$ (with $k$ letters) and the set $U$ formed by the following $2k^{2}+(m-1)k$ monoid relations:
 \begin{description}
\item[($U_1$)] $(c_{i}b_{i}^{m-1})^{2}=1$, $1\le i\le k$; 
\item[($U_2$)] $b_{i}^{m}=1$, $1\le i\le k$ ;
\item[($U_3$)] $(b_{i}c_{i}b_{i}^{m-1})^{m-1}=1$, $1\le i\le k$;
\item[($U_4$)] $(c_{i}b_{i}^{m-2}c_{i})^{3}=1$, $1\le i\le k$; 
\item[($U_5$)] $(c_{i}b_{i}^{m-j-1}c_{i}b_{i}^{j-1})^{2}=1$, $2\leq j\leq m-2$, $1\le i\le k$;
\item[($U_6$)] $b_ib_j=b_jb_i$, $c_ic_j=c_jc_i$, $1\le i < j\le k$; \quad 
$b_{i}c_{j}=c_{j}b_{i}$, $1\le i,j\le k$, $i\ne j$.
\end{description}
By Proposition \ref{dir}, the monoid $\Sym_{m}^{k}$ is also defined by the presentation $\langle B,C \mid U\rangle$.

\smallskip

Let us also consider the $k$-letters alphabet $D=\{d_i \mid 1\leq i\leq k\}$. 
Recall that, as elements of $\Sym_{m}^{k}$, we have $d_{i}=b_{i}c_{i+1}$, for $1\le i\le k-1$, and $d_{k}=b_{k}c_{1}$. 
Moreover, $c_{1}=d_{k}^{m}$, 
$c_{i+1}=d_{i}^{m}$, for $1\le i\le k-1$, and $b_{i}=d_{i}^{(m-1)^2}$, for $1\le i\le k$. 
Also, notice that $d_i^{m(m-1)}=1$, whence $b_i^{m-1}=d_i^{(m-1)^3}=d_i^{m-1}$, for $1\le i\le k$. 
By applying Tietze transformations to the previous presentation, it is easy to check that $\Sym_{m}^{k}$ is also defined 
by the presentation $\langle D \mid V\rangle$, where $V$ is formed by the following $2k^{2}+(m-2)k$ monoid relations: 
\begin{description}
\item[($V_1$)] $(d_{k}^{m}d_{1}^{m-1})^{2}=1$;\quad   $(d_{i}^{m}d_{i+1}^{m-1})^{2}=1$, $1\le i\le k-1$;

\item[($V_2$)] $d_{i}^{m(m-1)}=1$, $1\le i\le k$;

\item[($V_3$)] $(d_{1}^{(m-1)^2}d_{k}^{m}d_{1}^{m-1})^{m-1}=1$;\quad   
$(d_{i+1}^{(m-1)^2}d_{i}^{m}d_{i+1}^{m-1})^{m-1}=1$, $1\le i\le k-1$; 

\item[($V_4$)]  $(d_{k}^{m}d_{1}^{(m-1)^{2}(m-2)}d_{k}^{m})^{3}=1$;\quad   
$(d_{i}^{m}d_{i+1}^{(m-1)^{2}(m-2)}d_{i}^{m})^{3}=1$, $1\le i\le k-1$;  

\item[($V_5$)]  $(d_{k}^{m}d_{1}^{(m-1)^{2}(m-j-1)}d_{k}^{m}d_{1}^{(m-1)^{2}(j-1)})^{2}=1$, $2\leq j\leq m-2$; 

$(d_{i}^{m}d_{i+1}^{(m-1)^{2}(m-j-1)}d_{i}^{m}d_{i+1}^{(m-1)^{2}(j-1)})^{2}=1$, $2\leq j\leq m-2$, $1\le i\le k-1$; 

\item[($V_{6}$)] $d_i^{m}d_j^{m}=d_j^{m}d_i^{m},  d_i^{(m-1)^2}d_j^{(m-1)^2}=d_j^{(m-1)^2}d_i^{(m-1)^2}$, $1\le i<j\le k$; 

$d_{i}^{(m-1)^{2}}d_{k}^{m}=d_{k}^{m}d_{i}^{(m-1)^{2}}$, $2\le i\le k-1$; \quad 
$d_{i}^{(m-1)^{2}}d_{j}^{m}=d_{j}^{m}d_{i}^{(m-1)^{2}}$, $1\le i\le k$, $1\le j\le k-1$, $i\not\in\{j, j+1\}$. 
\end{description}

\smallskip 

We move on to the monoid $\POI_{k}$. Let $X=\{x_{i}\mid 0\le i\le k-1\}$ be an alphabet (with $k$ letters).  For $k\ge 2$, let $W$ be the set formed by the following $\frac{1}{2}(k^{2}+5k-4)$ monoid relations:
\begin{description}
\item[($W_1$)] $x_ix_0=x_0x_{i+1}$, $1\leq i\leq k-2$;
\item[($W_2$)] $x_jx_i=x_ix_j$, $2\leq i+1<j\leq k-1$;
\item[($W_3$)] $x_0^2x_1=x_0^2=x_{k-1}x_0^2$;
\item[($W_{4}$)] $x_{i+1}x_ix_{i+1}=x_{i+1}x_i=x_ix_{i+1}x_i$, $1\leq i\leq k-2$;
\item[($W_{5}$)] $x_ix_{i+1}\ldots x_{k-1}x_0x_1\ldots x_{i-1}x_i=x_i$, $0\leq i\leq k-1$;
\item[($W_{6}$)] $x_{i+1}\ldots x_{k-1}x_0x_1\ldots x_{i-1}x_i^2=x_i^2$, $1\leq i\leq k-1$.
\end{description}
The presentation $\langle X\mid W \rangle$ defines the monoid $\POI_{k}$ (see \cite{Fernandes:2001} or \cite{Fernandes:2002b}).

\smallskip  

Finally, we define three sets of relations that envolve the letters from $X$ together with the previous alphabets considered. 
Foremost, let $R'$ be the set formed by the following $2k^{2}+2k$ monoid relations over the alphabet $A\cup B\cup X$: 
\begin{description}
\item[($R'_{1}$)] $a_{1}x_0=x_0$, $b_{1}x_0=x_0$;
\item[($R'_{2}$)] $x_0a_{i}=a_{i+1}x_0$, $x_0b_{i}=b_{i+1}x_0$, $1\leq i\leq k-1$;
\item[($R'_{3}$)] $x_ia_{k-i}=x_i$, $x_ib_{k-i}=x_i$, $0\leq i\leq k-1$;
\item[($R'_{4}$)] $a_{i}x_{k-i+1}=x_{k-i+1}$,  $b_{i}x_{k-i+1}=x_{k-i+1}$, $2\leq i\leq k$;
\item[($R'_{5}$)] $x_ia_{k-i+1}=a_{k-i}x_i$, $x_ib_{k-i+1}=b_{k-i}x_i$, $1\leq i\leq k-1$;
\item[($R'_{6}$)] $x_ia_{j}=a_{j}x_i$, $x_ib_{j}=b_{j}x_i$, $1\leq i\leq k-1$, $1\le j\le k$, $j\notin \{k-i,k-i+1\}$. 
\end{description}
Secondly, consider the set $U'$ formed by the following $2k^{2}+2k$ monoid relations over the alphabet $B\cup C\cup X$: 
\begin{description}
\item[($U'_{1}$)] $c_{1}b_{1}^{m-1}x_0=x_0$, $b_{1}x_0=x_0$;
\item[($U'_{2}$)] $x_0c_{i}b_{i}^{m-1}=c_{i+1}b_{i+1}^{m-1}x_0$, $x_0b_{i}=b_{i+1}x_0$, $1\leq i\leq k-1$;
\item[($U'_{3}$)] $x_ic_{k-i}b_{k-i}^{m-1}=x_i$, $x_ib_{k-i}=x_i$, $0\leq i\leq k-1$;
\item[($U'_{4}$)] $c_{i}b_{i}^{m-1}x_{k-i+1}=x_{k-i+1}$,  $b_{i}x_{k-i+1}=x_{k-i+1}$, $2\leq i\leq k$;
\item[($U'_{5}$)] $x_ic_{k-i+1}b_{k-i+1}^{m-1}=c_{k-i}b_{k-i}^{m-1}x_i$, $x_ib_{k-i+1}=b_{k-i}x_i$, $1\leq i\leq k-1$;
\item[($U'_{6}$)] $x_ic_{j}b_{j}^{m-1}=c_{j}b_{j}^{m-1}x_i$, $x_ib_{j}=b_{j}x_i$, $1\leq i\leq k-1$, $1\le j\le k$, $j\notin \{k-i,k-i+1\}$. 
\end{description}
Lastly, let $V'$ be the set formed by the following $2k^{2}+2k$ monoid relations over the alphabet $D\cup X$: 
\begin{description}
\item[($V'_{1}$)] $d_{k}^{m}d_{1}^{m-1}x_0=x_0$, $d_{1}^{(m-1)^{2}}x_0=x_0$;

\item[($V'_{2}$)] $x_0d_{k}^{m}d_{1}^{m-1}=d_{1}^{m}d_{2}^{m-1}x_0$; \quad 
$x_0d_{i}^{m}d_{i+1}^{m-1}=d_{i+1}^{m}d_{i+2}^{m-1}x_0$, $1\le i\le k-2$;  

$x_0d_{i}^{(m-1)^{2}}=d_{i+1}^{(m-1)^{2}}x_0$, $1\leq i\leq k-1$;

\item[($V'_{3}$)] $x_id_{k-i-1}^{m}d_{k-i}^{m-1}=x_i$, $0\leq i\leq k-2$; \quad $x_{k-1}d_{k}^{m}d_{1}^{m-1}=x_{k-1}$;\quad 
$x_id_{k-i}^{{(m-1)}^{2}}=x_i$, $0\leq i\leq k-1$;
 
\item[($V'_{4}$)] $d_{i}^{m}d_{i+1}^{m-1}x_{k-i}=x_{k-i}$,  $d_{i+1}^{(m-1)^{2}}x_{k-i}=x_{k-i}$, $1\leq i\leq k-1$;

\item[($V'_{5}$)] $x_id_{k-i}^{m}d_{k-i+1}^{m-1}=d_{k-i-1}^{m}d_{k-i}^{m-1}x_i$, $1\leq i\leq k-2$; \quad 
$x_{k-1}d_{1}^{m}d_{2}^{m-1}=d_{k}^{m}d_{1}^{m-1}x_{k-1}$; 

$x_id_{k-i+1}^{(m-1)^{2}}=d_{k-i}^{(m-1)^{2}}x_i$, $1\leq i\leq k-1$; 

\item[($V'_{6}$)] $x_id_{k}^{m}d_{1}^{m-1}=d_{k}^{m}d_{1}^{m-1}x_i$, $1\le i\le k-2$; \quad 
$x_id_{j}^{m}d_{j+1}^{m-1}=d_{j}^{m}d_{j+1}^{m-1}x_i$, $1\leq i,j\leq k-1$, $j\notin \{k-i-1,k-i\}$; 

$x_id_{j}^{(m-1)^{2}}=d_{j}^{(m-1)^{2}}x_i$, $1\leq i\leq k-1$, $1\le j\le k$, $j\notin \{k-i,k-i+1\}$.
\end{description}

\medskip

Clearly, the presentations $\langle A\cup B\cup X\mid R\cup W\cup R' \rangle$, $\langle B\cup C\cup X\mid U\cup W\cup U' \rangle$ and $\langle D\cup X\mid V\cup W\cup V' \rangle$ can be obtained from each other by applying a finite number of Tietze transformations. Therefore,  
by Proposition \ref{tietze}, they define the same monoid. We will prove that they define the monoid $\POI_{k\times m}$ by showing that  
$\langle A\cup B\cup X\mid R\cup W\cup R' \rangle$ is a presentation for $\POI_{k\times m}$ in terms of its generators $\bar A\cup \bar B\cup \bar X$ defined in the previous section. The method described in Proposition \ref{met} will be used. 

\smallskip

Let $f:A\cup B\cup X\longrightarrow \POI_{k\times m}$ be the mapping defined by $a_{i}f=\bar{a}_{i}$, $b_{i}f=\bar{b}_{i}$ and $x_jf=\bar{x}_j$, $i\in\{1,2,\ldots,k\}$, $j\in\{0,1,\ldots,k-1\}$. Let $\varphi:{(A\cup B\cup X)}^*\longrightarrow \POI_{k\times m}$ be the natural homomorphism that extends $f$ to ${(A\cup B\cup X)}^*$. 

\medskip 

It is a routine matter to prove the following lemma: 

\begin{lemma}\label{gerrel} 
The generating set $\bar{A}\cup \bar{B}\cup \bar X$ of $\POI_{k\times m}$ satisfies (via $\varphi$) all the relations from $R\cup W\cup R'$. 
\end{lemma}

Observe that, it follows from the previous lemma that $w_1\varphi=w_2\varphi$, for all $w_1,w_2\in (A\cup B\cup X)^*$ such that 
$w_1=w_2$ is a consequence of $R\cup W\cup R'$. 

\smallskip 

\begin{lemma}\label{twoletters} 
Let $e\in A\cup B$ and $x\in X$. Then, there exists $f \in A\cup B\cup\{1\}$ such that the relation $ex=xf$ is a consequence of $R'$.
\end{lemma}
\begin{proof}
The result follows immediately from relations $(R'_1)$ and $(R'_2)$, for $x=x_0$, and from relations $(R'_4)$, $(R'_5)$ and $(R'_6)$, 
for $x\in X\setminus\{x_0\}$.  
\end{proof}

\smallskip 

Let us denote by $|w|$ the length of a word $w\in (A\cup B\cup X)^*$.

\begin{lemma}\label{dec} 
For each $w\in {(A\cup B\cup X)}^*$ there exist $u \in X^*$ and $s \in (A\cup B)^*$ such that the relation $w=us$ is a consequence of $R'$.
\end{lemma}
\begin{proof}
We will prove the lemma by induction on the length of $w\in {(A\cup B\cup X)}^*$.

Clearly, the result is trivial for any $w\in {(A\cup B\cup X)}^*$ such that $|w|\le1$. 

Let $t\ge2$ and, by induction hypothesis, admit the result for all $w\in {(A\cup B\cup X)}^*$ such that $|w|<t$. 

Let $w\in {(A\cup B\cup X)}^*$ be such that $|w|=t$. Then, there exist $v\in {(A\cup B\cup X)}^*$ and $x\in A\cup B\cup X$ 
such that $w\equiv vx$ and $|v|=t-1$. 

Since $|v|<t$, by induction hypothesis, there exist $u_1 \in X^*$ and $s_1 \in (A\cup B)^*$ such that the relation $v=u_1s_1$ is a consequence of $R'$. Hence the relation $w=u_1s_1x$ is also a consequence of $R'$. 

If $|s_1|=0$ or $x\in A\cup B$ then the result is proved. 

So, suppose that $|s_1|\ge1$ and $x\in X$. Let $s'\in (A\cup B)^*$ and $e\in A\cup B$ be such that $s_1\equiv s'e$. 
By Lemma \ref{twoletters}, there exists $f \in A\cup B\cup\{1\}$ such that the relation $ex=xf$ is a consequence of $R'$. 
On the other hand, since $|s'|<|s_1|\le |v|<t$, we have $|s'x|<t$ and so, by induction hypothesis, there exist $u_2 \in X^*$ 
and $s_2 \in (A\cup B)^*$ such that the relation $s'x=u_2s_2$ is a consequence of $R'$. Thus, the relation $w=u_1u_2s_2f$, where $u_1u_2\in X^*$ and $s_2f\in (A\cup B)^*$,  is also a consequence of $R'$, as required. 
\end{proof}

\begin{lemma}\label{uyu} 
For all $j\in\{1,2,\ldots,k\}$ and $u\in X^*$ such that $I_j\nsubseteq \im(u\varphi)$, the relations $ua_{j}=u$ 
and $ub_j=u$ are consequences of $R'$. 
\end{lemma}
\begin{proof}
We will prove the lemma for relations of the form $ua_j=u$, with $1\le j\le k$, by induction on the length of $u \in X^*$. 
For relations of the form $ub_j=u$, with $1\le j\le k$, the proof is analogous. 

Let $j\in\{1,2,\ldots,k\}$ and let $u\in X^*$ be such that $I_j\nsubseteq \im(u\varphi)$ and $|u|=1$. 
Then $u\equiv x_i$, for some $0\le i\le k-1$. As $I_j\nsubseteq \im(x_i\varphi)=\im(\bar{x}_i)$, 
by definition of $\bar{x}_i$, we have $j=k-i$. Hence, the relation $ua_j=u$, i.e. $x_ia_{k-i}=x_i$, is one of the relations $(R'_{3})$.

Next, let $t\ge1$ and, by induction hypothesis, admit that for all 
$j\in\{1,2,\ldots,k\}$ and $u\in X^*$ such that $I_j\nsubseteq \im(u\varphi)$ and $|u|=t$, the relation $ua_{j}=u$ 
is a consequence of $R'$.

Let $j\in\{1,2,\ldots,k\}$ and let $u\in X^*$ be such that $I_j\nsubseteq \im(u\varphi)$ and $|u|=t+1$. Let $\bar u=u\varphi$. 

Take $v\in X^*$ and $i\in\{0,1,\ldots,k-1\}$ such that $u\equiv vx_i$. Observe that $|v|=t$. Let $\bar v=v\varphi$. 
Hence $\bar u=\bar v\bar x_i$. 

First, consider $i=0$. If $j=k$ then $x_0a_k=x_0$ is a relation from $(R'_3)$ and so $vx_0a_k=vx_0$, 
i.e. $ua_j=u$, is a consequence of $R'$. On the other hand, suppose that $1\le j\le k-1$. 
If $I_{j+1}\subseteq\im(\bar v)$ then, as $I_{j+1}\bar x_0=I_j$, we obtain $I_j\subseteq \im(\bar v\bar x_0)=\im(\bar u)$, which is a contradiction. 
Hence, $I_{j+1}\nsubseteq\im(\bar v)$ and so, by induction hypothesis, the relation $va_{j+1}=v$ is a consequence of $R'$. 
Since $x_0a_j=a_{j+1}x_0$ is a relation from $(R'_2)$, we deduce $ua_j\equiv vx_0a_j=va_{j+1}x_0=vx_0\equiv u$, as a consequence of $R'$.

Now, suppose that $1\le i\le k-1$.  

If $j=k-i$ then $x_ia_j=x_i$ is a relation from $(R'_3)$ and so $vx_ia_j=vx_i$, 
i.e. $ua_j=u$, is a consequence of $R'$. 

Next, let $j=k-i+1$. 
If $I_{j-1}\subseteq\im(\bar v)$ then, as $I_{j-1}\bar x_i=I_{k-i}\bar x_i=I_j$, we have $I_j\subseteq \im(\bar v\bar x_i)=\im(\bar u)$, 
which is a contradiction. 
Hence, $I_{j-1}\nsubseteq\im(\bar v)$ and so, by induction hypothesis, the relation $va_{j-1}=v$ is a consequence of $R'$. 
Since $x_ia_j=a_{j-1}x_i$ is a relation from $(R'_5)$, we obtain $ua_j\equiv vx_ia_j=va_{j-1}x_i=vx_i\equiv u$, as a consequence of $R'$.

Finally, admit that $j\not\in\{k-i,k-i+1\}$. 
If $I_{j}\subseteq\im(\bar v)$ then, as $I_{j}\bar x_i=I_j$, we get $I_j\subseteq \im(\bar v\bar x_i)=\im(\bar u)$, 
which is a contradiction. 
Hence, $I_{j}\nsubseteq\im(\bar v)$ and so, by induction hypothesis, the relation $va_{j}=v$ is a consequence of $R'$ and so 
$ua_j\equiv vx_ia_j=va_{j}x_i=vx_i\equiv u$ are also consequences of $R'$, as required. 
\end{proof} 

\begin{lemma}\label{fcan} 
For each $w\in {(A\cup B\cup X)}^*$ there exist $u \in X^*$ and $s \in (A\cup B)^*$ such that the relation $w=us$ is a consequence of $R\cup R'$ and $\ell(s\varphi)=\ell$, for all $\ell\in\Omega_n\setminus\im(u\varphi)$.
\end{lemma}
\begin{proof}
Let $u \in X^*$ and $s_0 \in (A\cup B)^*$ be such that the relation $w=us_0$ is a consequence of $R'$ (by applying Lemma \ref{dec}). 
Hence, we may take $s \in (A\cup B)^*$ such that $w=us$ is a consequence of $R\cup R'$ and $s$ has minimum length among all 
$s' \in (A\cup B)^*$ such that $w=us'$ is a consequence of $R\cup R'$.  

Let $\ell\in\Omega_n\setminus\im(u\varphi)$. Then $\ell\in I_j$, for some $1\le j\le k$. 

Suppose that $a_j$ or $b_j$ occur in $s$. 
Let $s_1\in((A\cup B)\setminus\{a_j,b_j\})^*$, $y\in\{a_j,b_j\}$ and $s_2 \in (A\cup B)^*$ be such that $s\equiv s_1ys_2$. 
Then, clearly, the relation $s_1y=ys_1$ is a consequence of relations from $(R_6)$ and so the relation $s=ys_1s_2$ is a consequence of $R$. 
On the other hand, as $I_j\nsubseteq\im(u\varphi)$, by Lemma \ref{uyu}, the relation $uy=u$ is a consequence of $R'$, whence the relation $w=us_1s_2$ is a consequence of $R\cup R'$,  $s_1s_2 \in (A\cup B)^*$ and $|s_1s_2|=|s|-1$, which is a contradiction. 

Therefore, $a_j$ and $b_j$ do not occur in $s$ and so the restriction of $s\varphi$ to $I_j$ is the identity of $I_j$. 
In particular, $\ell(s\varphi)=\ell$, as required. 
\end{proof}

Let $u\in X^*$ and $s\in (A\cup B)^*$. Notice that, as $s\varphi\in\G_{k\times m}$, then $I_j(s\varphi)=I_j$, for all $1\le j\le k$, and so 
$\dom((us)\varphi)=\dom(u\varphi)$ and $\im((us)\varphi)=\im(u\varphi)$. 

\smallskip 

We are now in a position to prove our last lemma. 

\begin{lemma}\label{apres} 
Let $w_1,w_2\in (A\cup B\cup X)^*$. 
If $w_1\varphi=w_2\varphi$ then $w_1=w_2$ is a consequence of $R\cup W\cup R'$. 
\end{lemma}
\begin{proof}
By Lemma \ref{fcan} we can consider $u_1,u_2 \in X^*$ and $s_1,s_2 \in (A\cup B)^*$ such that the relations $w_1=u_1s_1$ and $w_2=u_2s_2$ are consequences of $R\cup R'$, $\ell(s_1\varphi)=\ell$, for all $\ell\in\Omega_n\setminus\im(u_1\varphi)$, and 
$\ell(s_2\varphi)=\ell$, for all $\ell\in\Omega_n\setminus\im(u_2\varphi)$. 

Observe that 
$
\dom(u_1\varphi)=\dom((u_1s_1)\varphi)= \dom(w_1\varphi)=\dom(w_2\varphi)=\dom((u_2s_2)\varphi)=\dom(u_2\varphi)
$
and 
$
\im(u_1\varphi)=\im((u_1s_1)\varphi)= \im(w_1\varphi)=\im(w_2\varphi)=\im((u_2s_2)\varphi)=\im(u_2\varphi). 
$
Since $u_1\varphi,u_2\varphi\in\POI_{k\times m}\cap \POI_n$, it follows that $u_1\varphi=u_2\varphi$. 
On the other hand, $\POI_{k\times m}\cap \POI_n\simeq \POI_k$ and 
the monoid $\POI_k$ is defined by the presentation $\langle B\mid W\rangle$. 
Therefore, the relation $u_1=u_2$ is a consequence of $W$. 
 
Next, we turn our attention to $s_1\varphi, s_2\varphi\in \G_{k\times m}$. 
Let $\ell\in\Omega_n\setminus\im(u_1\varphi)=\Omega_n\setminus\im(u_2\varphi)$. 
Then $\ell(s_1\varphi)=\ell=\ell(s_2\varphi)$. 
On the other hand, let $\ell\in\im(u_1\varphi)=\im(u_2\varphi)$. Take $t\in\Omega_n$ such that $t(u_1\varphi)=\ell$. 
Then $\ell(s_1\varphi)=(t(u_1\varphi))(s_1\varphi)=t((u_1\varphi)(s_1\varphi))=t((u_1s_1)\varphi)=
t(w_1\varphi)=t(w_2\varphi)=t((u_2s_2)\varphi)=t((u_2\varphi)(s_2\varphi))=(t(u_2\varphi))(s_2\varphi)=\ell(s_2\varphi)$. 
Hence $s_1\varphi=s_2\varphi$. 
Since $\G_{k\times m}\simeq\Sym_{m}^{k}$ and $\Sym_m^k$ is defined by the presentation $\langle A,B\mid R\rangle$, 
it follows that the relation $s_1=s_2$ is a consequence of $R$. 

Thus, the relation $u_1s_1=u_2s_2$ is a consequence of $R\cup W$ and so 
the relation $w_1=w_2$ is a consequence of $R\cup W\cup R'$, as required. 
 \end{proof}

In view of Proposition \ref{met}, it follows immediately from Lemmas \ref{gerrel} and \ref{apres}: 

\begin{theorem}
For $k,m\ge2$, the monoid $\POI_{k\times m}$ is defined by the presentation $\langle A\cup B\cup X\mid R\cup W\cup R' \rangle$ on $3k$ generators and $\frac{1}{2}(9k^2+(2m+7)k-4)$ relations.
\end{theorem}

In view of Proposition \ref{tietze}, as corollaries of the previous theorem, we also have:  

\begin{theorem}
For $k,m\ge2$, the monoid $\POI_{k\times m}$ is defined by the presentation $\langle B\cup C\cup X\mid U\cup W\cup U' \rangle$ on $3k$ generators and 
$\frac{1}{2}(9k^2+(2m+7)k-4)$ relations.
\end{theorem}

\begin{theorem}
For $k,m\ge2$, the monoid $\POI_{k\times m}$ is defined by the presentation $\langle D\cup X\mid V\cup W\cup V' \rangle$ on $2k$ generators and 
$\frac{1}{2}\left(9k^2+(2m+5)k-4\right)$ relations.
\end{theorem}


\bigskip 

\lastpage 

\end{document}